\documentclass[9pt,shortpaper,twoside,print]{IEEEtran}
\usepackage{silence}
\usepackage[noadjust]{cite}
\usepackage{amsmath,amssymb,amsfonts}
\usepackage{amsthm}
\usepackage{algorithmic}
\usepackage{algorithm}
\usepackage{graphicx}
\usepackage{graphbox}
\usepackage{textcomp}
\usepackage{multirow}
\usepackage{balance}
\usepackage{xcolor}
\usepackage[export]{adjustbox}
\usepackage{color}
\usepackage{accents}
\usepackage{textcomp}
\usepackage{comment}
\usepackage{bm}
\usepackage{tikz}
\let\labelindent\relax
\usepackage[shortlabels]{enumitem}
\def\BibTeX{{\rm B\kern-.05em{\sc i\kern-.025em b}\kern-.08em
    T\kern-.1667em\lower.7ex\hbox{E}\kern-.125emX}}

\definecolor{subsectioncolor}{rgb}{0,0.541,0.855}

\newcommand{\black}[1]{\color{black}{#1}\color{black}\phantom{}}
\definecolor{LightCyan}{rgb}{0.85,0.85,1.0}
\newtheorem{theorem}{Theorem}

\newtheorem{lemma}{Lemma}
\newtheorem{remark}{Remark}

\newtheorem{proposition}{Proposition}
\newtheorem{assumption}{Assumption}

\renewcommand{\det}{\text{ \normalfont{det}}}

\newcommand{\cov}{\normalfont\textsf{\footnotesize cov}}

\newcommand{\tr}{\normalfont\text{\normalfont tr}}
\newcommand{\mfs}[1]{{\normalfont\textsf{#1}}}
\newcommand{\mf}{\mathbf}

\DeclareMathOperator*{\argmin}{arg\,min}

\makeatletter
\let\NAT@parse\undefined
\makeatother
\usepackage[hidelinks]{hyperref}
\begin{document}

\title{Distributed Discrete-time Dynamic Outer Approximation \\ of the Intersection of Ellipsoids}
\author{Eduardo Sebasti\'{a}n$^*$, Rodrigo Aldana-L\'{o}pez$^*$, Rosario Aragü\'{e}s, Eduardo Montijano and Carlos Sagü\'{e}s
\thanks{This work was supported by ONR Global
grant N62909-24-1-2081, the Spanish projects PID2021-125514NB-I00, PID2021-124137OB-I00, TED2021-130224B-I00
funded by MCIN/AEI/10.13039/501100011033, by ERDF A way of making Europe and by the
European Union NextGenerationEU/PRTR, Project DGA T45-23R by the Gobierno de Aragón, by the Universidad de Zaragoza and Banco Santander, by CONACYT-Mexico grant 739841 and Spanish grant FPU19-05700.}%
\thanks{All the authors are with the Departamento de Inform\'atica e Ingenier\'ia de Sistemas and the Instituto de Investigaci\'on en Ingenier\'ia de Arag\'on, Universidad de Zaragoza, Spain 
(email:\texttt{\scriptsize  esebastian@unizar.es, rodrigo.aldana.lopez@gmail.com, raragues@unizar.es, emonti@unizar.es, csagues@unizar.es}). $^*$Equal contribution.}%
}

\newcommand\copyrighttext{%
  \footnotesize \textcopyright This paper has been accepted for publication at IEEE Transactions on Automatic Control. Please, when citing the paper, refer to the official manuscript with the following DOI: 10.1109/TAC.2025.3540423.}
\newcommand\copyrightnotice{%
\begin{tikzpicture}[remember picture,overlay]
\node[anchor=south,yshift=10pt] at (current page.south) {\fbox{\parbox{\dimexpr\textwidth-\fboxsep-\fboxrule\relax}{\copyrighttext}}};
\end{tikzpicture}%
}

\maketitle
\copyrightnotice

\begin{abstract}
This paper presents the first discrete-time distributed algorithm to track the tightest ellipsoids that outer approximates the global dynamic intersection of ellipsoids. Given an undirected network, we consider a setup where each node measures an ellipsoid, defined as a time-varying positive semidefinite matrix. The goal is to devise a distributed algorithm to track the tightest outer approximation of the intersection of all the ellipsoids. The solution is based on a novel distributed reformulation of the original centralized semi-definite outer L\"owner-John program, characterized by a non-separable objective function and global constraints. We prove finite-time convergence to the global minima of the centralized problem in the static case and finite-time bounded tracking error in the dynamic case. Moreover, we prove boundedness of estimation in the tracking of the global optimum and robustness in the estimation against time-varying inputs. 
We illustrate the properties of the algorithm with different simulated examples, including a distributed estimation showcase where our proposal is integrated into a distributed Kalman filter to surpass the state-of-the-art in mean square error performance.
\end{abstract}

\begin{IEEEkeywords}
Consensus, distributed optimization, discrete time systems, ellipsoidal methods
\end{IEEEkeywords}

\section{Introduction}\label{sec:intro}

\IEEEPARstart{F}{rom} safe control to distributed sensor fusion, one fundamental problem in control systems is how to approximate system measurements, states, or constraints such that their essential features are preserved, being simple enough to be handled. One of the most popular representations is ellipsoids~\cite{Lasserre2015Generalization}, characterized by symmetric and positive semi-definite $n$-dimensional matrices. Ellipsoids are chosen to describe the conservative shape of an obstacle~\cite{ malyuta2021advances}, obstacle-free regions \cite{deits2015computing}, estimated quantities with their associated uncertainty \cite{Sebastian2024TAC}, very large multi-agent populations~\cite{saravanos2023distributed}, or patterns and geometrical objects to be identified and clustered \cite{yang2022certifiably}. All these examples are related to or can be posed as finding the tightest outer ellipsoid approximating a convex set \cite{vandenberghe1996semidefinite}. 

Specifically, when the convex set is defined as the intersection of $\mathsf{N}$ ellipsoids, the outer Löwner-John method~\cite{Henrion2001LMI} provides an approximate solution derived from a rank constraint relaxation. No distributed method exists, yet, to compute the solution of the outer L\"owner-John method when the ellipsoids evolve with time despite its promising applications. Possible applications of such solution include (i) stochastic distributed estimation under unknown correlations where ellipsoids represent the covariance of the most accurate fusion  \cite{Sebastian2021CDC,Sebastian2024TAC}, (ii) robust cooperative control where ellipsoids represent a common safe region \cite{schwarting2021stochastic}, and (iii) computer vision where the resulting ellipsoid represents, e.g., the best clustering that separates two sets of data points in a collaborative task \cite{Lasserre2015Generalization}. 

Currently, only a continuous-time and static solution exists \cite{aldana2023distributed}, with no straightforward extension for the dynamic case. The approach taken in this work is radically different, by developing a distributed algorithm that works in discrete time (i.e., deployable in realistic settings) and under time-varying ellipsoids.

In this work, we propose the first distributed discrete-time dynamic algorithm to compute the solution of the original centralized outer L\"owner-John method, approximating the global solution to the ellipsoid intersection problem. From a distributed optimization point of view, compared to other works 
(Sec.~\ref{sec:related}), the L\"owner-John problem is comprised by a centralized optimization program with a non-separable objective function over positive definite matrices and with global coupling constraints (Sec.~\ref{sec:problem}), preventing the use of current distributed optimization methods. The proposed algorithm is based on a novel distributed reformulation of the outer L\"owner-John method  (Sec.~\ref{sec:distributed}). Each node solves a local semi-definite program whilst the agreement value tracks the dynamic intersection of ellipsoids. We analyze the theoretical properties of the algorithm (Sec. \ref{sec:static}). For constant input ellipsoids, we prove finite-time convergence to the global minima of the centralized problem; for time-varying input ellipsoids, we prove finite-time bounded tracking error. We prove that our algorithm guarantees boundedness of the estimates and robustness on the feasibility of global tracking under dynamic input ellipsoids. We illustrate the proposed algorithm in simulated experiments (Sec.~\ref{sec:examples}), including an integration on a distributed Kalman filter for stochastic estimation (Sec.~\ref{sec:application}).

\section{Related Work}\label{sec:related}

The outer L\"owner-John method is a semi-definite program \cite{Henrion2001LMI} with three main characteristics: (i) the objective function is not necessarily separable, so it cannot be expressed as the sum of local functions; (ii) some of the constraints couple all the optimization variables; and (iii) the optimization variables are symmetric positive definite variables. All together, these features lead to a challenging problem from a distributed algorithmic perspective.

Many distributed optimization solutions depart from a centralized optimization formulation and then derive consensus-based algorithms to reconstruct global quantities and converge to the global minimizer \cite{notarstefano2019distributed}. When the objective function is separable as a sum of local (strongly) convex smooth functions, dynamic consensus over the optimization variables and the gradients \cite{carnevale2023triggered} can be exploited to design distributed versions of gradient descent \cite{Nedic2014Distributed}, general first order optimization methods \cite{ van2022universal} or second order {Newton-Raphson}-like methods \cite{varagnolo2015newton}. Despite proving the success of consensus-based approaches for distributed optimization, these solutions are not suitable  when the objective function cannot be decomposed into local smooth convex functions. 

To deal with constraints, some works propose extensions of consensus-based algorithms that rely on projection methods \cite{mai2023distributed}, population dynamics equations \cite{martinez2022discrete}, or subgradient methods \cite{romao2021subgradient}. The most popular alternative to address global coupling constraints is the use of a primal and/or dual proximal stage at each node \cite{notarnicola2019constraint}. In particular, the alternating direction method of multipliers \cite{bastianello2022novel} and the Douglas-Rachford splitting \cite{bredies2022graph} allow to handle equality coupling constraints among nodes by solving simultaneously the  primal and dual optimization problem. Typically, these distributed proximal methods either assume separability of the original objective function \cite{meng2015proximal} and/or are restricted to scalar quantities \cite{parikh2014block}. Therefore, they are not suitable for distributing the outer L\"owner-John method. In the presence of constraints and a linear objective function, it is possible to find distributed semi-definite problem reformulations \cite{li2021distributed} that are solved through primal-dual methods. With a similar spirit to our work, \cite{burger2013polyhedral} proposes a polyhedral outer approximation in a distributed optimization context. Nevertheless, in this case the polyhedral approximation represents a convex constraint set rather than the optimization objective. 

To address the non-separability of the objective function \cite{yang2019survey}, some works assume that the optimization variables can be divided in two subsets \cite{wendell1976minimization}, which is not possible in the outer L\"owner-John method since it entails an ellipsoidal object and not scalars. Thus, it is common to develop ad hoc solutions that are based on successive approximations \cite{scutari2019distributed} or decompositions \cite{meselhi2022decomposition}, again not suitable for symmetric positive semi-definite matrix variables. 

{From an application point of view, the outer L\"owner-John method has been widely used for stochastic distributed estimation under the name of Covariance Intersection (CI) \cite{julier1997non}. CI is a method to consistently integrate estimates of neighboring nodes in networks with unknown cross-correlations. CI is optimal for two nodes \cite{Reinhardt2015CI}, but in general it is suboptimal \cite{cros2023optimality}. The consistency property motivates the use of CI for distributed Kalman filtering \cite{wei2018stability, yan2024distributed}. Nevertheless, despite improved performance upon classical distributed Kalman filters \cite{Olfati2007DKF}, the lack of optimality guarantees in the fusion leads to suboptimality for these filtering alternatives. Other extensions of CI include the Covariance Union (CU) \cite{uhlmann2003covariance} and the Split Covariance Intersection (SCI) \cite{Julier2001CI}; the former determines the tightest ellipsoid that guarantees consistency regardless of which of the estimates is consistent, while the latter splits the admissible set of cross-covariances to tighten the conservative bound provided by CI. They have both being applied in filtering problems \cite{reece2010generalised,li2013cooperative}. In pursuit of optimality guarantees in stochastic distributed estimation, we recently used the original outer L\"owner-John method to derive the certifiable optimal distributed Kalman filter under unknown correlations \cite{Sebastian2021CDC, Sebastian2024TAC}. These works restrict the outer L\"owner-John method to the local neighborhood of each node, rather than reconstructing the global optimum provided by all the node estimates. For this reason, and to demonstrate the benefits of our proposed algorithm, we integrate the solution in a distributed Kalman filter.}

\section{Problem statement}\label{sec:problem}

{\textbf{Notation: } $\tr(\bullet), \det(\bullet)$ denote trace and determinant. Let $\mathbb{R}$ be the real number set. $\mathbb{S}^n,\mathbb{S}^n_{++},\mathbb{S}^n_+\subset\mathbb{R}^{n\times n}$ denote the sets of symmetric, positive definite and positive semi-definite matrices respectively. We use $\mf{0}, \mf{I}$ for the zero and identity matrices of appropriate dimensions. We denote by $\mf{0}\preceq \mf{P}$ when a matrix $\mf{P}\in\mathbb{R}^{n\times n}$ is positive semi-definite. Let $\mfs{relint}(\bullet)$ and $\mfs{rebdr}(\bullet)$ represent the relative interior and relative boundary operators respectively. Let $\mfs{co}(\mathcal{C})$ denote the convex hull of a set $\mathcal{C}$.}

Consider a network of $\mfs{N}$ agents which communicate according to an undirected graph $\mathcal{G}=(\mathcal{I},{\mathcal{F}})$, where \mbox{$\mathcal{I} = \{1,\dots,\mfs{N}\}$} is the vertex set and ${\mathcal{F}}\subset \mathcal{I}\times\mathcal{I}$ is the edge set representing the communication links between agents. The set of neighbors of agent $i$, including $i$, is $\mathcal{N}_i = \{j \in \mathcal{I} | (i,j) \in {\mathcal{F}}\}\cup\{i\}$. 

At each discrete-time step $k\in\{0,1,\dots\}$, each agent $i\in\mathcal{I}$ measures a positive definitive \textit{input} matrix $\mf{P}_i[k]\in\mathbb{S}^n_{++}$. To provide some intuition, this matrix can be pictured as representing the uncertainty over an estimate \cite{Olfati2007DKF, Julier2017General}, where $\mf{P}_i[k]^{-1}$ represents its information matrix. Such information matrix characterizes the ellipsoid \mbox{$\mathcal{E}(\mf{P}_i[k]^{-1})=\{\mf{y}\in\mathbb{R}^n : \mf{y}^\top\mf{P}_i[k]^{-1}\mf{y}\leq 1\}$}. In practice, the input matrices $\mf{P}_i[k]$ cannot have arbitrarily large values, and the variation of $\mf{P}_i[k]$ from instant $k$ to $k+1$ cannot be arbitrarily large, so in this work we do the following assumption. 
\begin{assumption}\label{as:bounded} {There exist $0 < \underline{p}\leq  \overline{p}$
 such that the input matrices $\mf{P}_i[k]$, $\forall i\in\mathcal{I}$ and $\forall k\in\{0,1,\dots\}$, satisfy
$$
\begin{aligned}
\underline{p}\mf{I}&\preceq \mf{P}_i[k]\preceq \overline{p}\mf{I
}.\\
%\underline{\theta}\mf{P}_i[k-1]&\preceq \mf{P}_i[k]\preceq  \bar{\theta}\mf{P}_i[k-1]\\  %$0\leq\underline{\theta}\leq1\leq \bar{\theta}$
\end{aligned}
$$
The previous inequalities provide bounds for the eigenvalues of $\mf{P}_i[k]$ for all time.}% and relative to the previous instant respectively.
\end{assumption}

The goal for all the agents is to agree on a matrix $\mf{Q}^*[k]$ that represents the tightest outer approximation of the intersection of the ellipsoids $\mathcal{E}(\mf{P}_i[k]^{-1})$, as shown in Fig. \ref{fig:problem_setting}. Intuitively, $\mf{Q}^*[k]$ represent the information matrix for such covering ellipsoid. 
Formally, $\mf{Q}^*[k]$ is the result of the following optimization problem:
\begin{equation}
 \label{eq:prob}
\begin{aligned}
\mf{Q}^*[k] &= \argmin_{\mf{Q}\in\mathcal{C}^*[k]} f(\mf{Q}) \\
\mathcal{C}^*[k] = \bigg\{&\mf{Q}\in\mathbb{S}_+^n : \exists \lambda_1,\dots,\lambda_{\mfs{N}}\in[0,1],\\& \sum_{j=1}^\mfs{N}\lambda_j \leq 1,  \mf{0}\preceq \mf{Q} \preceq \sum_{j=1}^\mfs{N} \lambda_j \mf{P}_j[k]^{-1} \bigg\}
\end{aligned}
\end{equation}
\begin{figure}
    \centering
    \includegraphics[width=0.8\columnwidth]{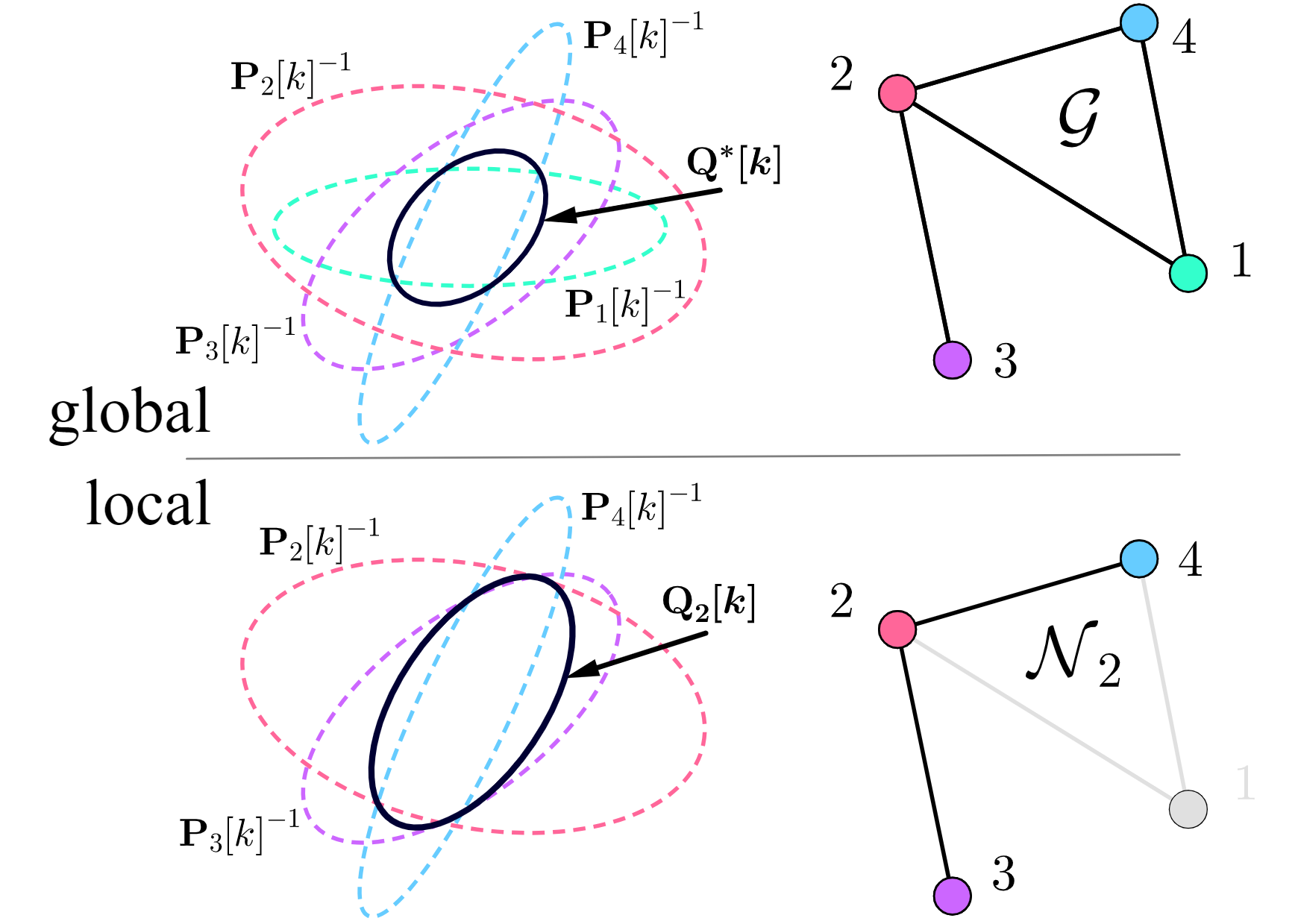}
    \caption{Problem setting: (top) each node in graph $\mathcal{G}$ represents a time-varying ellipsoid described by the matrix $\mathbf{P}_i[k]  \quad \forall i \in \{1, 2, 3, 4\}$. The goal is to cooperate to find the smallest ellipsoid that outer approximates their intersection (black ellipsoid for $\mathbf{Q}^*[k]$); (bottom) each node (e.g., node $2$) only has access to its neighboring information, leading to conservative approximations of $\mathbf{Q}^*[k]$ (black ellipsoid for $\mathbf{Q}_2[k]$). In this work we propose a distributed method to make $\mathbf{Q}_i[k]$ converge to $\mathbf{Q}^*[k]$ using local information.  }
    \label{fig:problem_setting}
\end{figure}
where we make the following assumption on the function $f$: 
\begin{assumption}
\label{as:unbounded}
The function $f:\mathbb{S}_{+}^n\to\mathbb{R}$ satisfies:
\begin{enumerate}
\item $f$ has continuous derivative  over all $\mathbb{S}_+^n$.
\item $f$ is {strictly} convex, meaning that, for any
$c_1,c_2 {\geq} 0,$ \mbox{$ \mf{Q}_1,\mf{Q}_2\in\mathbb{S}_{+}^n$} with $c_1+c_2=1$ it follows that:
\mbox{$
f(c_1\mf{Q}_1 + c_2\mf{Q}_2)< c_1f(\mf{Q}_1)+c_2f(\mf{Q}_2)
$}.
\item $f$ has no lower bound over $\mathbb{S}_+^n$, meaning that $\inf_{\mf{Q}\in\mathbb{S}_{+}^n} f(\mf{Q})$ does not exists.
\item $f$ is bounded over any closed bounded convex set $\mathcal{C}\subset\mathbb{S}_{+}^n$ meaning that $\inf_{\mf{Q}\in\mathcal{C}} f(\mf{Q})\in\mathbb{R}$.
\end{enumerate}
\end{assumption}
Assumption \ref{as:unbounded} holds for many popular choices of $f(\bullet)$ related to the volume of ellipsoids or information theoretic measures, such as $\log(\det((\bullet)^{-1}))$ or $\tr((\bullet)^{-1})$. Moreover, Assumption \ref{as:unbounded}-4) is an instrumental property to be used in subsequent analysis.

The program in \eqref{eq:prob} is known as the outer L\"owner-John method {and the intuition behind its structure is the following. The minimization of the objective function seeks to minimize the size of the output ellipsoid. For a problem with $\mathsf{N}=2$ nodes, the first constraint of $\mathcal{C}^*[k]$ means that, if $\lambda_1=0$, then $\mathcal{E}(\mathbf{Q}_1) \subseteq \mathcal{E}(\mathbf{Q}_2)$, i.e., the ellipsoid at node $1$ is contained in the ellipsoid at node $2$. In general, $\lambda_1, \lambda_2 > 0$, which means that the two ellipsoids have an overlapping area in common, characterized by the convex combination of $\mf{P}_1[k]^{-1}, \mf{P}_2[k]^{-1}$, which is the second constraint in $\mathcal{C}^*[k]$. The ellipsoid $\mf{Q}^*[k]$ is the tightest ellipsoid that covers such overlapping area. The same reasoning extends to the $\mathsf{N}>2$ case.}

The formulation of the outer L\"owner-John method is centralized, in the sense that all the inputs $\mf{P}_i[k]$ are needed to compute $\mf{Q}^*[k]$. However, in the aforementioned cooperative tasks, it is of key importance to develop a distributed solution such that each agent has a local version $\mf{Q}_i[k]$ of $\mf{Q}^*[k]$ computed from local interactions. 

%To do so, in Section~\ref{sec:distributed} we present an algorithm that computes the the solution of the global outer L\"owner-John method in a distributed manner at each node. 

\section{Distributed Discrete-time Outer Ellipse Computation}\label{sec:distributed}

To compute the global optimum of problem~\eqref{eq:prob}, we propose that, at each instant $k$, each agent $i$ solves the semi-definite program:
\begin{equation}
\label{eq:algo}
\begin{aligned}
\mf{Q}_i[k] &= \argmin_{\mf{Q}\in\mathcal{C}_i[k]} f(\mf{Q}) \\
\mathcal{C}_i[k] = \bigg\{&\mf{Q}\in\mathbb{S}_+^n : \exists \lambda_j^i,\lambda_\mf{P}^i\in[0,1], j\in\mathcal{N}_i, \lambda_\mf{P}^i + \sum_{j\in\mathcal{N}_i}^\mfs{N}\lambda_j^i \leq 1,
\\& \mf{0}\preceq \mf{Q} \preceq \lambda_\mf{P}^i\mf{P}_i[k]^{-1} + \frac{1}{\overline{\theta}}\sum_{j\in\mathcal{N}_i} \lambda_j^i \mf{Q}_j[k-1] \bigg\}
\end{aligned}
\end{equation}
{for $k\geq 1$}, with initial conditions $\mf{Q}_i[0] = \mf{P}_i[0]^{-1}$ {and a parameter $\overline{\theta}$ to be specified subsequently}. Note that each node only requires the solution of \eqref{eq:algo} at each neighbor in the previous instant. Therefore, \eqref{eq:algo} defines a distributed algorithm, described in Algorithm~\ref{al:algorithm}. Its properties are summarized in Theorem \ref{th:main}.

\begin{algorithm}
\caption{Distributed estimation of \eqref{eq:prob} at node $i$}\label{al:algorithm}
\begin{algorithmic}[1]
\STATE Initialization: $\lambda^i_j, \lambda_{\mathbf{P}}^i = 0\quad \forall j \in \mathcal{N}_i$, $\bar{\theta}_i = \bar{\theta} \geq 1$ and  ${\mathbf{Q}}_i[0] = \mathbf{P}_i[0]^{-1}$ 
\FOR{$k = 1, 2, \hdots$}
    \STATE Information exchange:
            \begin{itemize}
               \item[] Send $\mathbf{Q}_i[k-1]$ to neighbors $j \in \mathcal{N}_i$.
               \item[] Receive $\mathbf{Q}_j[k-1]$ from neighbors $j \in \mathcal{N}_i$.
           \end{itemize}
    \STATE Solve local reformulated outer L\"owner-John program: 
    \begin{itemize}
        \item[] $\mf{Q}_i[k] = \argmin_{\mf{Q}\in\mathcal{C}_i[k]} f(\mf{Q})$, with $\mathcal{C}_i[k]$ given in \eqref{eq:algo}.
    \end{itemize}
\ENDFOR
\end{algorithmic}
\end{algorithm}

\begin{theorem}
\label{th:main}
Let Assumptions \ref{as:bounded} and \ref{as:unbounded} hold and $\mathcal{G}$ be connected. Then, {for a network that executes Algorithm \ref{al:algorithm}, }the following statements are true:
\begin{enumerate}
    \item (Robustness)
    $\mathcal{C}_i[k]\subseteq \mathcal{C}^*[k]$ for all $k\in\{{1,2}\dots\}, i\in\mathcal{I}$.
    \item (Boundedness) $\mf{Q}_i[k]\preceq {(1/\underline{p})}\mf{I}$ for all $k\in\{{1,2}\dots\}, i\in\mathcal{I}$.
    \item (Convergence) Given any choice of norm $\|\bullet\|$ in $\mathbb{S}^n_{+}${ and any $\delta>0$, there exists $K\in\{{1,2},\dots\}, 0\leq\underline{\theta}<1< \bar{\theta}$ such that if
    $
    \underline{\theta}\mf{P}_i[k-1] \preceq \mf{P}_i[k]\preceq  \bar{\theta}\mf{P}_i[k-1], \forall k=1,2,\dots
    $
    then,}
$$
{\| \mf{Q}_i[k] - \mf{Q}^*[k] \| \leq \delta \text{ for any }k\geq K, i\in\mathcal{I}.} 
$$
{In addition, if $\underline{\theta}=\bar{\theta}=1$, then $\mf{Q}_i[k]=\mf{Q}^*[k]$.}
    
\end{enumerate}

\end{theorem}

\begin{figure}
    \centering
    \includegraphics[width=0.9\columnwidth]{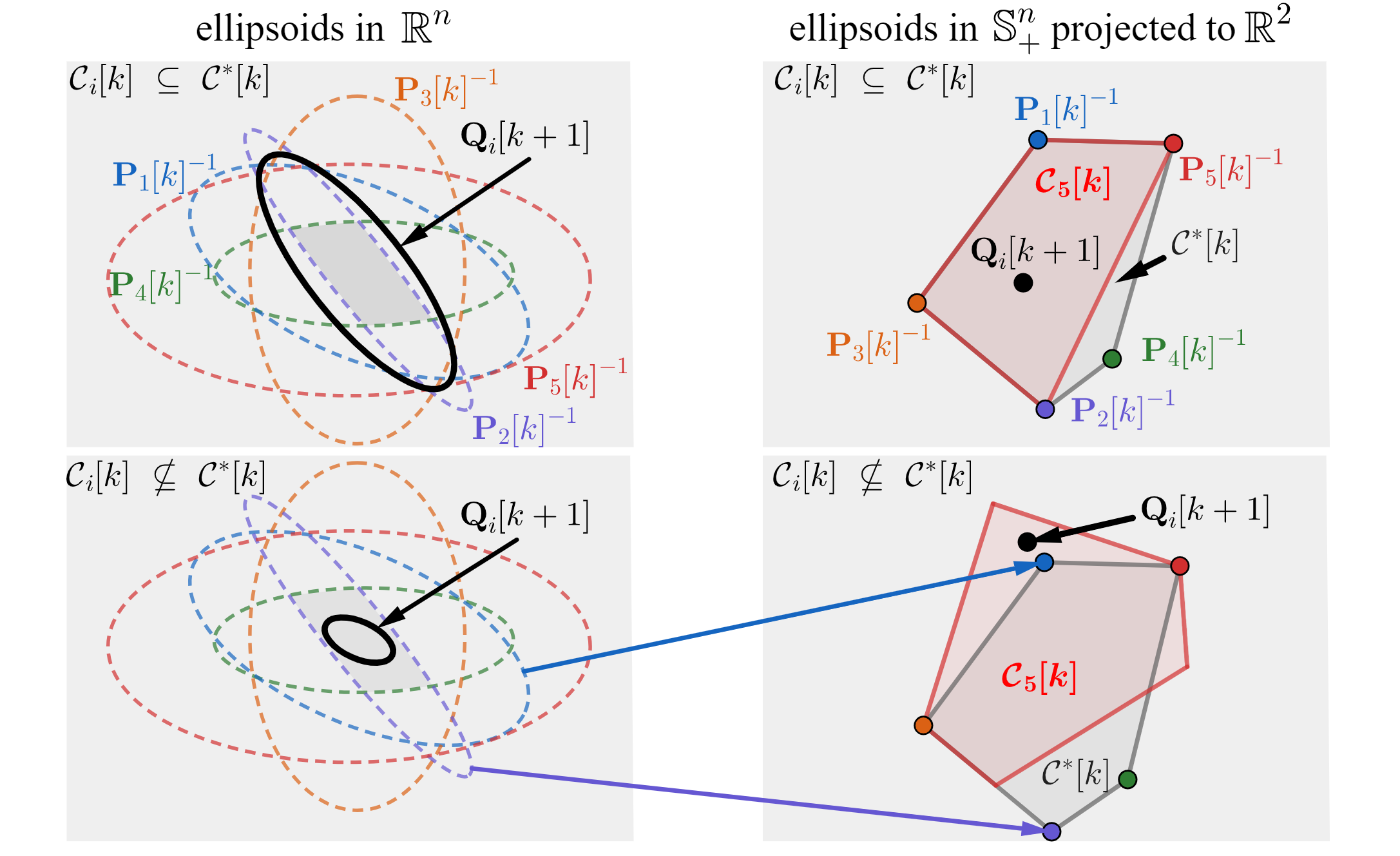}
    \caption{An example of the importance of \textit{robustness} (Theorem~\ref{th:main}) in our problem. Throughout the paper, we combine the representation of the ellipsoids in $\mathbb{R}^n$ (left column) and in $\mathbb{S}^n_{+}$ projected to $\mathbb{R}^2$ (right column). The latter helps understanding the relationships between the input ellipsoids $\mathbf{P}_i[k]^{-1}$, their local outer L\"owner-John solutions $\mathbf{Q}_i[k+1]$ and the global optimum $\mathbf{Q}^*[k]$.
    When $\mathcal{C}_i[k] \subseteq \mathcal{C}^*[k]$ (top), the local estimate $\mathbf{Q}_i[k+1]$ is always contained in both the local and global feasible set, which implies that the node is able to reconstruct an approximation of $\mathbf{Q}^*[k]$ and the global intersection of ellipsoids, even if $f(\mathbf{Q}_i[k-1]) < f(\mathbf{Q}^*[k])$. On the other hand, when $\mathcal{C}_i[k] \nsubseteq \mathcal{C}^*[k]$ (bottom), the local optimum estimate $\mathbf{Q}_i[k+1]$ might be out of the global feasible set. As shown on the right, if $f(\mathbf{Q}_i[k]) < f(\mathbf{Q}^*[k+1])$, then the node is not able to track the outer approximation of the intersection because the current estimate $\mathbf{Q}_i[k]$ is smaller than $\mathbf{Q}^*[k+1]$ in the sense of $f$. Moreover, $\mathbf{Q}_i[k]$ will propagate across the network, blocking all nodes from correctly estimating $\mathbf{Q}^*[k+1]$ at subsequent time steps.}
    \label{fig:robustness}
\end{figure}

We prove Theorem \ref{th:main} in Section \ref{subsec:dynamic}, after introducing auxiliary results in Sections \ref{subsec:aux} and \ref{subsec:constant}. Before that, we describe the main properties expressed in the theorem. The intuition behind \textit{robustness} is that, given arbitrary input matrices $\{\mathbf{P}_i[k]\}_{i=1}^{\mathsf{N}}$ under Assumption~\ref{as:bounded}, the feasible set of the local optimization problem at each node is always contained inside the feasible set of the global optimization problem; this is important because, if \textit{robustness} did not hold, then the optimum of the local optimization problem at node $i$ might be such that $\mathbf{Q}_i[k] \notin \mathcal{C}^*[k]$ and, therefore, it would not be possible to converge arbitrarily close to the global optimum. For instance, Fig.~\ref{fig:robustness} visualizes a case where, if $f(\mathbf{Q}_i[k-1]) < f(\mathbf{Q}^*[k])$ $\forall k>K$ and \textit{robustness} did not hold, then $\mathbf{Q}_i[k] = \mathbf{Q}_i[k-1]$ $\forall k>K$, which means that the estimate at node $i$ would get stuck in a potentially unfeasible point forever. Thus, the proposed algorithm is robust against changes in the input ellipsoids, being able to converge to the global optimum always. Second, \textit{boundedness} implies that, irrespective of the changes in the input ellipsoids, the estimate of the global optimum at each node never escapes to infinity. %; this is important because, otherwise, the estimated ellipsoid at node $i$ and instant $K$ might converge to a point and, after that, \eqref{eq:algo} would be ill-posed and the solutions for all $k>K$ would not be able to be recovered. 
Third, \textit{convergence} means that our algorithm converges in finite time $K$ to a region around the optimum of the global problem, where the size of the region depends on how fast the input ellipsoids vary with time: in particular, if the input ellipsoids are constant, then $\delta = 0$ and the global optimum is perfectly recovered; besides, the estimates remain inside that region for $k>K$.

\begin{remark}
    In many applications, the inputs $\{\mathbf{P}_i[k]\}_{i=1}^{\mathsf{N}}$ come from the discretization of the continuous-time dynamics of $\{\mathbf{P}_i(t)\}_{i=1}^{\mathsf{N}}$. In this case, the designer can choose the sampling step sufficiently small, such that Assumption \ref{as:bounded} holds for \mbox{$\underline{\theta} \simeq \bar{\theta} \simeq 1$} and a desired accuracy $\delta$ is achieved.
\end{remark}

% \begin{remark}
% If $n=1$, ellipsoids reduce to non-negative scalars and the problem can be posed as (adaptive) dynamic min/max consensus \cite{deplano2018discrete,Deplano2022Dynamic, lippi2022adaptive} under the constraint that the time-varying input scalars are always non-negative. Hence,~\eqref{eq:algo} is a generalization of (adaptive) dynamic min/max consensus to positive semi-definite matrices.
% \end{remark}

\begin{remark}\label{remark:adaptive}
    The tracking error depends on the rate of change of the input ellipsoids. This suggests an adaptive scheme over parameter $\bar{\theta}$. One can follow similar procedures to those found in \cite{lippi2022adaptive, Deplano2022Dynamic} but adapted to positive definite matrices. The rate of change of local input matrices at instant $k$ is approximately given by \mbox{$\mathbf{P}_i[k-1]\mathbf{P}_i[k]^{-1}$}. Then, in practice, one can choose $\bar{\theta}_i[k] = \kappa \mathbf{P}_i[k-1]\mathbf{P}_i[k]^{-1}$ with $\kappa > 0$ in place of $\bar{\theta}$ in \eqref{eq:algo}. 
\end{remark}

After discussing the proposed algorithm and its main properties, the next section is devoted to prove Theorem~\ref{th:main}. 
%We use symbol $=$ to denote equivalent matrices, i.e., $\mathbf{Q}_1 = \mathbf{Q}_2$ means that matrices $\mathbf{Q}_1$ and $\mathbf{Q}_2$ are both solution of the same optimization problem.

\section{Convergence analysis}\label{sec:static}

To prove Theorem~\ref{th:main}, we first present some auxiliary results (Sec.~\ref{subsec:aux}) helpful for the main proofs. After that, we prove the theorem for the static case, i.e., $\mathbf{P}_i[k] = \mathbf{P}_i$ $\forall i,k$ (Sec. \ref{subsec:constant}). This intermediate step provides the tools and insights to prove Theorem~\ref{th:main} in Section \ref{subsec:dynamic}.

\subsection{Auxiliary results}\label{subsec:aux}
First, we characterize some properties of general closed convex sets that are useful for the subsequent study of the properties of the feasible sets $\mathcal{C}_i[k]$ and $\mathcal{C}^*[k]$.

\begin{lemma}\label{le:char}
Let $\mathcal{C}_1, \mathcal{C}_2$ two closed and convex sets such that $\mathcal{C}_1\subseteq\mathcal{C}_2\subseteq\mathbb{S}_+^n$. Moreover, denote by
$$
\begin{aligned}
\mf{Q}_{1} = \argmin_{\mf{Q}\in\mathcal{C}_1} f(\mf{Q}) \qquad \text{and} \qquad
\mf{Q}_{2} = \argmin_{\mf{Q}\in\mathcal{C}_2} f(\mf{Q}).
\end{aligned}
$$
Then, the following holds:
\begin{enumerate}
    \item If $\mf{Q}_{2}\in\mathcal{C}_1$ then $\mf{Q}_1= \mf{Q}_2$. 
    \item If $\mf{Q}_2\notin\mathcal{C}_1$ then $\mf{Q}_1\in\mfs{rebdr}(\mathcal{C}_1)$.
\end{enumerate}
\end{lemma}
\begin{proof}
    For the first item, since $f(\bullet)$ {is strictly convex} and $\mathcal{C}_1,\mathcal{C}_2$ are convex, then $\mf{Q}_1,\mf{Q}_2$ must be unique. Hence, if $\mf{Q}_2\in\mathcal{C}_1$, then $\mf{Q}_1= \mf{Q}_2$ must follow. For the second item, $\mathcal{C}_1$ is closed, so the relative boundary $\mfs{rebdr}(\mathcal{C}_1)$ exists and contains all its relative boundary points. The next steps of the proof follow by contradiction. Assume that $\mf{Q}_1\in\mfs{relint}(\mathcal{C}_1)$. Then, there exist a ball $\mathcal{B}\in\mfs{relint}(\mathcal{C}_1)$ centered at $\mf{Q}_1$ such that $f(\mf{Q}_1)\leq f(\mf{Q}), \forall \mf{Q}\in\mathcal{B}$. Thus, $\mf{Q}_1$ is a local optimum of $f(\bullet)$, and, as a result of convexity of $f(\bullet)$ and $\mathcal{C}_1,\mathcal{C}_2$, it is the global optimum of $f(\bullet)$ in any case. By assumption of the lemma, $\mf{Q}_1\in\mathcal{C}_1\subseteq\mathcal{C}_2$ and thus $\mf{Q}_1=\mf{Q}_2$. However, by assumption of the second item, $\mf{Q}_2\notin\mathcal{C}_1$, which leads to the contradiction $\mf{Q}_1\notin\mathcal{C}_1$.
\end{proof}

The next lemma {demonstrates} properties of the feasible set $\mathcal{C}_i[k]$ and $\mathcal{C}^*[k]$ that are important to study the relationships between the centralized and distributed optimization problems~\eqref{eq:prob} and \eqref{eq:algo}.   
\begin{lemma}
\label{le:properties}
    Let Assumptions \ref{as:bounded} and \ref{as:unbounded} hold. Then, the following statements are true:
    \begin{enumerate}
        \item\label{it:regular} The sets $\mathcal{C}_i[k], \mathcal{C}^*[k]$ are closed and convex for all \mbox{$k\in\{{1,2},\dots\}$} and $ i\in\mathcal{I}$.

        \item\label{it:extrema} Denote $\mf{Q}^*[k]=\min_{\mf{Q}\in\mathcal{C}^*[k]}f(\mf{Q})$ for arbitrary fixed \mbox{$k\in\{0,1,\dots,\}$}. Then, $\mf{Q}^*[k]\in\mfs{co}\{\mf{P}_i[k]^{-1}\}_{i=1}^\mfs{N}$.   
        \item\label{it:no:escapes} $
\mathcal{C}_i[k]\subseteq \mathcal{C}^*[k]
$
for all $k\in\{{1,2},\dots\}, i\in\mathcal{I}$.
    \end{enumerate}
\end{lemma}
\begin{proof}
    For item \ref{it:regular}), the result is straightforward since the sets come from the standard semi-definite programs defined in~\eqref{eq:prob}-\eqref{eq:algo}. For item \ref{it:extrema}), let $\mathcal{C}_1=\mathcal{C}^*[k], \mathcal{C}_2=\mathbb{S}_+^n$. Hence, item 3) of Assumption \ref{as:unbounded} implies that $\mf{Q}_2$ from Lemma \ref{le:char} does not exist and as a result $\mf{Q}_2\notin\mathcal{C}_1$ and $\mf{Q}^*[k]=\mf{Q}_1\in\mfs{rebdr}(\mathcal{C}_1)$. Note that points in the relative boundary of $\mathcal{C}_1$ are convex combinations of $\{\mf{P}_i[k]^{-1}\}_{i=1}^\mfs{N}$ from which the result follows.

    For item \ref{it:no:escapes}) we proceed by induction. As induction base, set an arbitrary agent $i\in\mathcal{I}$ {and define $\mathcal{C}_i[0]$ only for $k=0$ as $\mathcal{C}_i[0]:=\{\mf{P}_i[{0}]^{-1}\}$, so that the initial condition comply $\mf{Q}_i[0] = \mf{P}_i^{-1}[0]\in\mathcal{C}_i[0]$}. Now, referring to the centralized optimization problem~\eqref{eq:prob}, let $\lambda_j=1$ if $j=i$ and $\lambda_j=0$ otherwise. Hence, $\mf{Q}_i[0]=\sum_{j=1}^\mfs{N}\lambda_j\mf{P}_j[{0}]^{-1}\in\mathcal{C}^*[{0}]$. Henceforth, $\mathcal{C}_i[0]\subseteq \mathcal{C}^*[0]$.

%     Besides, any feasible solution of problem~\eqref{eq:algo} $\mf{Q}\in\mathcal{C}_i[1]$ is such that the following condition holds
% $$
% 0\preceq \mf{Q}\preceq \lambda_{\mf{P}}^i\mf{P}_i[1]^{-1} + \sum_{j\in\mathcal{N}_i}\lambda^i_j\mf{P}_j[1]^{-1}
% $$
% Now, referring to the centralized optimization problem~\eqref{eq:prob}, let  
% $$
% \lambda_j = \begin{cases}
% \lambda_\mf{P}^i & \text{if } j=i\\
% \lambda_j^i & \text{if }j\in\mathcal{N}_i\\
% 0 & \text{otherwise}
% \end{cases}.
% $$
% Therefore, 
% $$
% \begin{aligned}
% &\sum_{i=1}^\mfs{N} \lambda_j \mf{P}_i[1]^{-1} = \lambda_{\mf{P}}^i\mf{P}_i[1]^{-1} + \sum_{j\in\mathcal{N}_i}\lambda^i_j\mf{P}_j[1]^{-1} \succeq \mf{Q}, \\
% &\sum_{i=1}^\mfs{N}\lambda_i = \lambda_\mf{P}^i + \sum_{j\in\mathcal{N}_i}\lambda_j^i\leq 1.
% \end{aligned}
% $$
% Therefore, the feasible solution set for agent $i$ of problem~\eqref{eq:algo} is a particular case of the feasible solution set of problem~\eqref{eq:prob} at the first instant. Hence, $\mf{Q}\in\mathcal{C}^*[1]$ and, as a result, \mbox{$\mathcal{C}_i[1]\subseteq\mathcal{C}^*[1]$}. 

Now, assume that $\mathcal{C}_i[k-1]\subseteq \mathcal{C}^*[k-1]$ for some \mbox{$k\in\{1,2,\dots\}$}. Then, it follows $\mf{Q}_i[k-1]\in\mathcal{C}^*[k-1]$, meaning that there exists $\beta_1^j,\dots,\beta_\mfs{N}^j\in[0,1]$ with \mbox{$\sum_{\ell=1}^\mfs{N}\beta_\ell^j\leq 1$} and
$$
\mf{0}\preceq \mf{Q}_j[k-1] \preceq \sum_{\ell=1}^\mfs{N}\beta_\ell^j\mf{P}_\ell[k-1]^{-1}\preceq \bar{\theta} \sum_{\ell=1}^\mfs{N}\beta_\ell^j\mf{P}_\ell[k]^{-1}
$$
For an arbitrary $\mf{Q}\in\mathcal{C}_i[k], {k=1,2,\dots}$ it follows that
$$
\begin{aligned}
\mf{Q}&\preceq \lambda_\mf{P}^i\mf{P}_i[k]^{-1} + \frac{1}{\bar{\theta}}\sum_{j\in\mathcal{N}_i} \lambda_j^i \mf{Q}_j[k-1] \\
&\preceq \lambda_\mf{P}^i\mf{P}_i[k]^{-1} + \sum_{j\in\mathcal{N}_i}\lambda_j^i\sum_{\ell=1}^\mfs{N}\beta_\ell^j\mf{P}_\ell[k]^{-1}
\end{aligned}
$$
Now, consider
$$
\lambda_\ell = \begin{cases}
\lambda_\mf{P}^i + \sum_{j\in\mathcal{N}_i}\lambda_j^i\beta^j_l& \text{if }\ell=i \\
\sum_{j\in\mathcal{N}_i}\lambda_j^i\beta^j_\ell & \text{otherwise}
\end{cases}
$$
Then, it follows that
$$
\begin{aligned}
\sum_{\ell=1}^\mfs{N}\lambda_\ell &= \lambda_\mf{P}^i + \sum_{j\in\mathcal{N}_i}\lambda_j^i{\beta_i^j} + \sum_{l\neq i}\sum_{j\in\mathcal{N}_i}\lambda_j^i\beta^j_\ell = \\
&=\lambda_\mf{P}^i + \sum_{j\in\mathcal{N}_i}\lambda_j^i\sum_{\ell=1}^\mfs{N}\beta^j_\ell 
\leq \lambda_\mf{P}^i + \sum_{j\in\mathcal{N}_i}\lambda_j^i\leq 1
\end{aligned}
$$
and
$$
\begin{aligned}
&\sum_{\ell=1}^\mfs{N} \lambda_\ell \mf{P}_\ell[k]^{-1} =\\&\left(\kern -0.1cm\lambda_\mf{P}^i + \sum_{j\in\mathcal{N}_i}\lambda_j^i\beta^j_i\right)\mf{P}_i[k]^{-1}+ \sum_{l\neq i} \left(\sum_{j\in\mathcal{N}_i}\lambda_j^i\beta^j_\ell\right)\mf{P}_\ell[k]^{-1} \\
&=\lambda_\mf{P}^i\mf{P}_i[k] + \sum_{\ell=1}^\mfs{N} \left(\sum_{j\in\mathcal{N}_i}\lambda_j^i\beta^j_\ell\right)\mf{P}_\ell[k]^{-1} \succeq \mf{Q}
\end{aligned}
$$
Therefore, the feasible solution set for agent $i$ of problem~\eqref{eq:algo} is a particular case of the feasible solution set of problem~\eqref{eq:prob}, $\mf{Q}\in\mathcal{C}^*[k]$ and, as a result, $\mathcal{C}_i[k]\subseteq\mathcal{C}^*[k]$.
\end{proof}

\subsection{Constant inputs}\label{subsec:constant}
In this section we focus on the static case, i.e., the input ellipsoids do not change with time. Henceforth, it is assumed that Assumption \ref{as:bounded} holds with \mbox{$\underline{\theta}={\overline{\theta}}=1$}, meaning that the inputs $\mf{P}_i[k]$ are constant.
\begin{lemma}
\label{le:constant:convergence}
Let Assumptions \ref{as:bounded} and \ref{as:unbounded} hold and $\underline{\theta}={\overline{\theta}}=1$. Then, 
\begin{equation}
\label{eq:decrease}
f(\mf{Q}_i[k])\leq f(\mf{Q}_i[k-1]).
\end{equation} 
for all $i\in\mathcal{I}, {k=1,2,\dots}$. Moreover, there exists ${f_1,\dots,f_\mfs{N}}\in\mathbb{R}$ such that $\lim_{k\to \infty}f(\mf{Q}_i[k])={f_i}$.
\end{lemma}
\begin{proof}
Let $\lambda_j^i=1$ if $j=i$ and $\lambda_j^i=0$ otherwise, and $\lambda_\mf{P}^i=0$. Therefore,
$
\lambda_\mf{P}^i\mf{P}_i[k]^{-1} + \sum_{j\in\mathcal{N}_i} \lambda_j^i \mf{Q}_j[k-1] = \mf{Q}_i[k-1]
$
and
$
\lambda_\mf{P}^i+\sum_{j\in\mathcal{N}_i}\lambda_j^i = 1.
$
Hence, $\mf{Q}_i[k-1]\in\mathcal{C}_i[k]$. Therefore, \eqref{eq:decrease} follows by noting that both $\mf{Q}_i[k],\mf{Q}_i[k-1]\in\mathcal{C}_i[k]$ but that $\mf{Q}_i[k]$ is the minimizer of $f(\bullet)$ over $\mathcal{C}_i[k]$. 

For the last part of the lemma, note that $\mathcal{C}^*$ is closed and convex due to Lemma \ref{le:properties}-\ref{it:regular}) such that item 4) of Assumption \ref{as:unbounded} implies $f$ is lower bounded over $\mathcal{C}^*$. Moreover, since $\mathcal{C}_i{[k]}\subseteq\mathcal{C}^*$ then $f$ attains the same lower bound over $\mathcal{C}_i{[k]}$. Furthermore, combine \eqref{eq:decrease} with such lower bound to conclude that ${f_i}:=\lim_{k\to\infty}f(\mf{Q}_i[k])$ must exist {due to monotonicity in \eqref{eq:decrease}}. 
\end{proof}
%\begin{lemma}\label{le:consensus}
%Let Assumptions \ref{as:bounded} and \ref{as:unbounded} hold and %${\bar{\theta}}=1$. Moreover, let $\mathcal{G}$ be connected. Then, %there exists $\mf{Q}_{\mfs{eq}}\in\mathbb{S}_+^n$ such that %$\lim_{k\to\infty}\mf{Q}_i[k]=\mf{Q}_{\mfs{eq}}$ for all %$i\in\mathcal{I}$.
%\end{lemma}
\begin{figure*}
    \centering
    \includegraphics[width=0.7\textwidth]{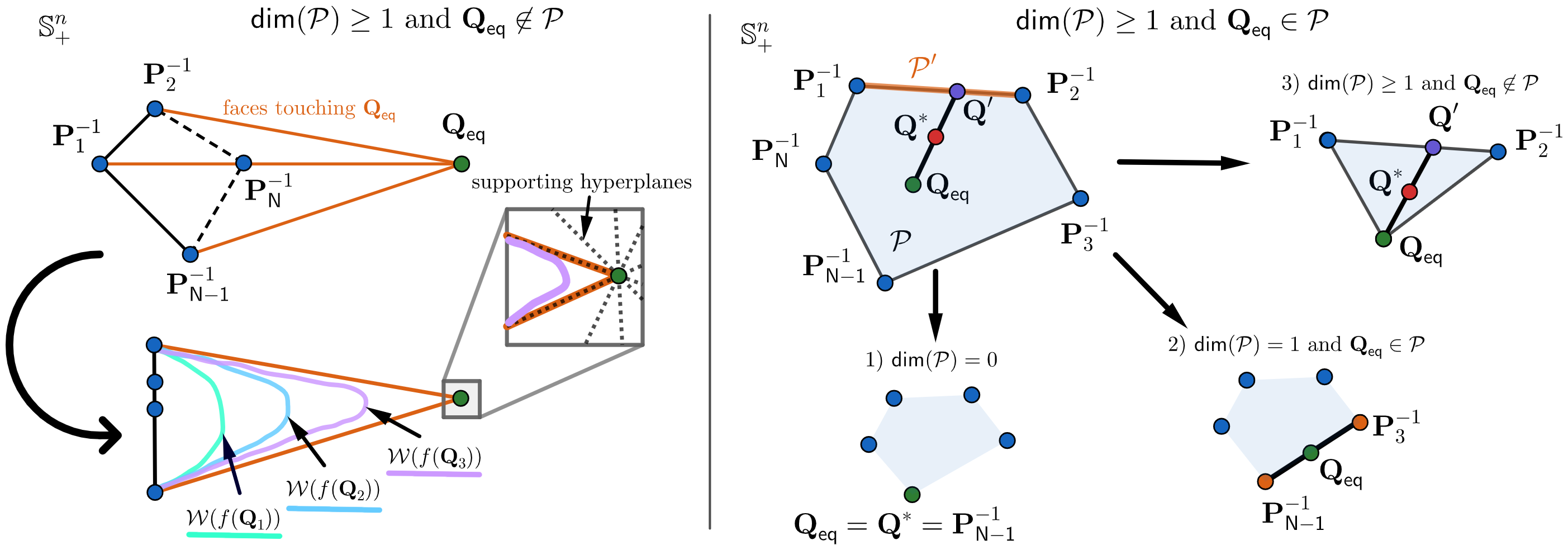}
    \caption{Illustrative visualization of the cases 3) and 4) of the proof of Lemma \ref{lem:global:optim}. The figure leverages the representation of ellipsoids in $\mathbb{S}^n_{+}$ projected to $\mathbb{R}^3$ and $\mathbb{R}^2$ introduced in Fig.~\ref{fig:robustness}.  }
    \label{fig:depiction_proof}
\end{figure*}
%\begin{proof}

%\end{proof}

%Convergence to an equilibrium does not immediately imply convergence to the global optimum, since the feasible sets at each node do not coincide in general. Therefore, we need to prove that $\mathbf{Q}_{\mathsf{eq}} = \mathbf{Q}^*$. 

\begin{lemma}\label{lem:global:optim}
{Let Assumptions \ref{as:bounded} and \ref{as:unbounded} hold, and $\underline{\theta}={\overline{\theta}}=1$. Moreover, let $\mathcal{G}$ be connected. Then, $\lim_{k\to\infty}\mf{Q}_i[k]=\mf{Q}^*$ for all $i\in\mathcal{I}$ where $\mf{Q}^*$ is the (constant) optimum of \eqref{eq:prob}.}
\end{lemma}
\begin{proof} 
\black{First,} we study the equilibrium of \eqref{eq:algo} as follows. Let $\ell = \argmin_{i\in\mathcal{I}} f_i^*$
with $f_i^*$ taken from Lemma  \ref{le:constant:convergence}. Moreover, let $\mathcal{N}_\ell^{r} = \{i\in\mathcal{I}: i\in\mathcal{N}_j, j\in\mathcal{N}_\ell^{r-1}\}$ and $\mathcal{N}_\ell^0 = \mathcal{N}_\ell$. This is, $\mathcal{N}_\ell^{r-1}$ contains the set of neighbors of neighbors of $\ell\in\mathcal{I}$ after $r$ hops. We now show that ${f_\ell=f_i}$ for all $i\in\mathcal{N}_\ell^r$ and all $r\in\{0,1,\dots\}$ by induction, for some $k$ in which equilibrium is attained. For the induction base, note that there exists ${\mf{Q}_\ell'}\in\mathbb{S}_+^n$ such that $f({\mf{Q}_\ell'})={f_\ell}\leq {f_i}\leq f(\mf{Q})$ for all $\mf{Q}\in\mathcal{C}_i[k]$ with $i\in\mathcal{N}_\ell$. Thus, \eqref{eq:algo} implies ${f_i}=f({\mf{Q}_\ell'})$ since ${\mf{Q}_\ell'}\in\mathcal{C}_i[k]$.  The induction step follows in the same way. Finally, since $\mathcal{G}$ is connected, then $\mathcal{N}_\ell^d = \mathcal{I}$ where $d$ is the diameter of $\mathcal{G}$. Therefore, ${f_i}={f_{\ell}}$ for all $i\in\mathcal{I}$. Finally, note that \eqref{eq:algo} {being a strictly convex program implies that $f(\mf{Q})={f_\ell}$ is attained for a single $\mf{Q}\in\mathbb{S}_+^n$ which we denote as $\mf{Q}_\mfs{eq}$.}

Moreover, item \ref{it:no:escapes}) of Lemma \ref{le:properties} implies that in equilibrium, $\mathcal{C}_i\subseteq\mathcal{C}^*$, which are constant $\mathcal{C}_i=\mathcal{C}_i[k], \mathcal{C}^*=\mathcal{C}^*[k]$ for $k\geq 1$. Now, assume $\mf{Q}^*\notin\mathcal{C}_i$ for all $i$ for a contradiction. Item 2) of Lemma \ref{le:char} implies $\mf{Q}_{\mfs{eq}}\in\mfs{rebdr}(\mathcal{C}_i)$. In equilibrium, it follows that
$$
\mfs{rebdr}(\mathcal{C}_i) =  \{\mf{Q}\in\mathbb{S}_+^n: \lambda_\mf{P}\mf{P}_i^{-1} + \lambda_i^{\mfs{eq}}\mf{Q}_{\mfs{eq}},\lambda_\mf{P}+\lambda_i^\mfs{eq}=1 \}
$$
since $\mf{Q}_i[k] = \mf{Q}_{\mfs{eq}}$ and with constant $\mf{P}_i=\mf{P}_i[k]$. This is, $\mf{Q}_{\mfs{eq}}$ is at the intersection of the $\mfs{N}$ lines $\{\mfs{rebdr}(\mathcal{C}_i)\}_{i\in\mathcal{I}}$. Denote the star $\mathcal{S}=\bigcup_{i\in\mathcal{I}}\{\mfs{rebdr}(\mathcal{C}_i)\}$ for the rest of the proof, and note that $\mf{Q}_\mfs{eq}=\argmin_{\mf{Q}\in\mathcal{S}}f(\mf{Q})$. On the other hand, note that, item \ref{it:extrema}) of Lemma \ref{le:properties} implies $\mf{Q}^*\in\mathcal{P}:=\mfs{co}\{\mf{P}_i^{-1}\}_{i=1}^\mfs{N}$. Denote with $\mfs{dim}(\mathcal{P})$ the dimension of the manifold $\mathcal{P}\subset \mathbb{S}_+^n$. Using these properties, we distinguish the following cases:

1) $\mfs{dim}(\mathcal{P})=0$, with $\mathcal{P}$ consisting of a single point $\mf{Q}^*=\mf{P}_1^{-1}=\dots=\mf{P}_\mfs{N}^{-1}$. However, note that from \eqref{eq:algo} it follows that $\mf{Q}_i[k]=\mf{Q}_\mfs{eq}=\mf{P}_i^{-1}{=\mf{Q}^*}$ {and thus $\mf{Q}^*\in\mathcal{C}_i$, which is a contradiction.}

2) $\mfs{dim}(\mathcal{P})=1$ and $\mf{Q}_\mfs{eq}\in\mathcal{P}$. In this case, $\mathcal{P}$ necessarily consists of a line segment from $\mf{P}_i^{-1}$ and $\mf{P}_j^{-1}$ with $i\neq j$. Note that in this case, the line segment must comply $\mathcal{P}=\mathcal{S}$ for the star $\mathcal{S}$. However, this imply $f(\mf{Q}_\mfs{eq})\leq f(\mf{Q}^{*})$. {Due to strong convexity of $f$, either  $f(\mf{Q}_\mfs{eq})< f(\mf{Q}^{*})$ which is impossible, or $\mf{Q}^*=\mf{Q}_\mfs{eq}=\mf{Q}_i[k]$ which leads to a contradiction similar to case 1)}.

3) $\mfs{dim}(\mathcal{P})\geq 1$ and $\mf{Q}_\mfs{eq}\notin\mathcal{P}$. In this case, denote with $\mathcal{H}=\mfs{co}(\mathcal{P}\cup\{\mf{Q}_\mfs{eq}\})$. Note that $\mathcal{H}$ is a polytope over the (possibly) lower dimensional space $\mfs{span}\{\mf{P}_1^{-1},\dots, \mf{P}_\mfs{N}^{-1}, \mf{Q}_\mfs{eq}\}\subset \mathbb{S}_+^n$, with $\mf{Q}_\mfs{eq}$ as one corner, with at least two faces touching it due to the dimension of $\mathcal{P}$. Let $\mathcal{W}(\alpha):=\{\mf{Q}\in\mathbb{S}_+^n: f(\mf{Q})\leq \alpha\}$ the level sets of $f$, which are convex for any $\alpha\in\mathbb{R}$. Moreover, note that $\mathcal{W}(f(\mf{Q}_\mfs{eq}))\subseteq \mathcal{C}^*\cap \mathcal{H}$. Otherwise, there would exist smaller $\alpha<f(\mf{Q}_\mfs{eq})$ such that $\mathcal{W}(f(\mf{Q}_\mfs{eq}))\cap \mathcal{S}\neq \varnothing$, which contradicts the fact that $f(\mf{Q}_\mfs{eq})\leq f(\mf{Q})$ for all $\mf{Q}\in\mathcal{S}$. Using the previous fact in combination with $\mf{Q}_\mfs{eq}\in\mathcal{W}(f(\mf{Q}_\mfs{eq}))$ hence, $\mathcal{W}(f(\mf{Q}_\mfs{eq}))$ has at least two supporting hyper-planes at $\mf{Q}_\mfs{eq}$, namely the faces of $\mathcal{H}$ meeting at that point. However,  $f$ is smooth by Assumption \ref{as:unbounded}, which imply that $\mathcal{W}(f(\mf{Q}_\mfs{eq}))$ must have a single supporting hyperplane at all points (see \cite[Theorems 23.3 and 25.1]{rockafellar1970}), {contradicting the previous fact that $\mathcal{W}(f(\mf{Q}_\mfs{eq}))$ has two supporting hyper-planes}.

4) $\mfs{dim}(\mathcal{P})> 1$ and $\mf{Q}_\mfs{eq}\in\mathcal{P}$. Set $\mf{Q}'\in\mfs{rebdr}(\mathcal{P})$ be the point such that $\mf{Q}^*\in\mathcal{P}$ is contained in the line segment between $\mf{Q}'$ and $\mf{Q}_\mfs{eq}$. Pick an arbitrary face $\mathcal{P}'\subseteq \mfs{rebdr}(\mathcal{P})$ of the polytope $\mathcal{P}$ with $\mf{Q}'\in\mathcal{P}'$, which is of lower dimension than $\mathcal{P}$ since $\mfs{dim}(\mfs{rebdr}{(\mathcal{P})})<\mfs{dim}(\mathcal{P})$. Hence, distinguish the same cases 1), 2), 3) and 4), with $\mathcal{P}$ replaced by the face $\mathcal{P}'$, $\mathcal{H}$ with $\mathcal{H}'=\mfs{co}(\mathcal{P}'\cup\{\mf{Q}_{\mfs{eq}}\})$ and noting that $\mf{Q}^*\in\mathcal{H}'$ allowing to follow the same reasoning, reaching a contradiction in cases 1), 2) and 3) directly. Case 4) is used recursively until other cases are reached, which always happens since dimension of $\mathcal{P}$ is decreased every recursion step.

Henceforth, a contradiction is reached in any case implying that $\mf{Q}^*\in\mathcal{C}_i$ for some $i\in\mathcal{I}$. Then, item \ref{it:regular}) of Lemma \ref{le:properties} allows the usage of item 1) of Lemma \ref{le:char}, which implies $\mf{Q}_i^*=\argmin_{\mf{Q}\in\mathcal{C}_i}f(\mf{Q})=\mf{Q}^*$, being $\mf{Q}_i^*=\mf{Q}_{\mfs{eq}}$ the unique equilibrium of \eqref{eq:algo}.
\end{proof}

\subsection{Dynamic inputs}
\label{subsec:dynamic}

\black{In this section}, we provide a proof for Theorem \ref{th:main} in the general case with dynamic inputs. %We prove each item of the theorem in the following:

 \textbf{Item 1)}: The result follows from item 3) of Lemma \ref{le:properties}.

 \textbf{Item 2)}: We proceed by induction. For the induction base, note that $\mf{Q}_i[0]=\mf{P}_i[0]^{-1}$ from which the bound follows directly by Assumption \ref{as:bounded}. Now, assume the bound for arbitrary time $k-1\in\{0,1,\dots\}$. Item 3) of Assumption \ref{as:unbounded} and item 2) of Lemma \ref{le:char} imply that $\mf{Q}_i[k]\in\mfs{rebdr}(\mathcal{C}_i[k])$. As a result
    % $$
    % \begin{aligned}
    % \mf{Q}_i[k]=&\lambda_\mf{P}^i\mf{P}_i[k]^{-1} + \frac{1}{\overline{\theta}}\sum_{j\in\mathcal{N}_i}\lambda_j^i\mf{Q}_j[k]\\\preceq & \overline{p}\bigg(\lambda_\mf{P}^i + \frac{1}{\overline{\theta}}\sum_{j\in\mathcal{N}_i}\lambda_j^i\bigg)\mf{I} \preceq \overline{p}\mf{I}
    % \end{aligned}
    % $$

        $$
    \begin{aligned}
    \mf{Q}_i[k] \kern -0.1cm = \kern -0.1cm \lambda_\mf{P}^i\mf{P}_i[k]^{-1} \kern -0.1cm + \kern -0.1cm \frac{1}{\overline{\theta}} \kern -0.1cm \sum_{j\in\mathcal{N}_i} \kern -0.1cm \lambda_j^i\mf{Q}_j[k \kern -0.1cm - \kern -0.1cm 1] \kern -0.0cm \preceq \kern -0.0cm {\frac{1}{\underline{p}}}\bigg(\kern -0.1cm\lambda_\mf{P}^i \kern -0.1cm + \kern -0.1cm {\frac{1}{\overline{\theta}}} \kern -0.1cm \sum_{j\in\mathcal{N}_i} \kern -0.1cm \lambda_j^i\bigg)\mf{I} \kern -0.0cm\preceq\kern -0.0cm {\frac{1}{\underline{p}}}\mf{I}
    \end{aligned}
    $$
\textbf{Item 3)}: First, let $\varepsilon,\underline{\theta},\overline{\theta}$ be such that $[\underline{\theta},\overline{\theta}]\subseteq[1-\varepsilon,1+\varepsilon]$. {Note that if $\varepsilon=0$, then $\overline{\theta}=\underline{\theta}=1$ which corresponds to the static case. Hence, we expect that small $\delta>0$ will lead to a small $\varepsilon>0$, preventing it to be arbitrarily big. To quantify this, we now show} the existence of $K\geq {1}, \varepsilon>0$ such that $\|\mf{Q}_i[K]-\mf{Q}^{*}[K]\|< \delta$ independently of the initial conditions, by repeated use of a continuity argument. We denote the trajectories of a nominal version of the algorithm \eqref{eq:algo} with constant inputs $\mf{P}_i[k]=\mf{P}_i[0]$ as $\mf{Q}_i^{\varepsilon=0}[k], {k\geq 1}$. Due to stability and optimally of the nominal system established in Lemma \ref{lem:global:optim}, for any $\delta_1>0$ there exist $K'\geq 0$ such that $\|\mf{Q}_i^{\varepsilon=0}[K']-\mf{Q}^{*,\varepsilon=0}[K']\|< \delta_1$. Moreover, note that $(1-\varepsilon)^{K'}\mf{P}_i[0]\preceq \mf{P}_i[K']\preceq (1+\varepsilon)^{K'}\mf{P}_i[0]$ by Assumption \ref{as:bounded}. Henceforth, for any $\delta_2>0$, there exists $\varepsilon>0$ sufficiently small to make $\mf{P}_i[K']$ as close as desired to $\mf{P}_i[0]$ to make $d_H(\mathcal{C}^*[K'], \mathcal{C}^{*,\varepsilon=0}[K'])< \delta_2$, where $d_H$ denotes the Hausdorff distance and $\mathcal{C}^{*,\varepsilon=0}[K]$ the feasible set in \eqref{eq:prob} in the nominal case. Similarly, for any $\delta_3>0$ there exists $\varepsilon>0$ such that $d_H(\mathcal{C}_i[K'], \mathcal{C}_i^{\varepsilon=0}[K'])<\delta_3$. Henceforth, by picking $\delta_1,\delta_2,\delta_3$ appropriately, $\varepsilon>0$ exist to make $\mf{Q}_i[K']$ as close as desired to $\mf{Q}_i^{\varepsilon=0}[K']$ (with error related to $\delta_3$), to $\mf{Q}^{*,\varepsilon=0}[K']$ (with error related to $\delta_1$) and to $\mf{Q}^{*}[K']$ (with error related to $\delta_2$) in that order, allowing $\|\mf{Q}_i[K']-\mf{Q}^{*}[K']\|< \delta$ for some $\delta$. 

Note that such $\varepsilon, K'$ depend on the initial conditions $\mf{Q}_i[0]$ since they might be different between trajectories. Make this dependence explicit in $K'=K'(\mf{Q}_1[0],\dots,\mf{Q}_\mfs{N}[0])$ and note that $\overline{K}(\delta):=\sup K'(\mf{Q}_1[0],\dots,\mf{Q}_\mfs{N}[0])$ exists since $\mf{Q}_1[0],\dots,\mf{Q}_\mfs{N}[0]$ lie in a compact set defined by \mbox{$\mf{Q}_i[0]=\mf{P}_i[0]^{-1}\preceq \underline{p}\mf{I}$}. A similar reasoning allows to conclude that  $\overline{\varepsilon}(\delta)=\inf\varepsilon(\mf{Q}_1[0],\dots,\mf{Q}_\mfs{N}[0])>0$ exists. Henceforth, for any $0<\varepsilon<\overline{\varepsilon}(\delta)$ it follows $\|\mf{Q}_i[\overline{K}(\delta)]-\mf{Q}^{*}[\overline{K}(\delta)]\|< \delta$ regardless of the initial conditions.

Assume that $\|\mf{Q}_i[k]-\mf{Q}^{*}[k]\|< \alpha$ holds for some arbitrary ${k\geq 0}$ and $\alpha<\delta$. Now, we show $\|\mf{Q}_i[k+1]-\mf{Q}^{*}[k+1]\|< \alpha$. The property follows by making $\varepsilon$ and $\alpha>0$ sufficiently small by taking $\mf{Q}^*[k]$ as close as desired to $\mf{Q}^*[k+1]$ and $\mf{Q}_i[k]$ close to $\mf{Q}_i[k+1]$, using a similar reasoning as before. Hence, there exists $\tilde{\varepsilon}=\inf\varepsilon(\mf{Q}_1[k],\dots,\mf{Q}_\mfs{N}[k])>0$ and $\tilde{\alpha}=\inf\alpha(\mf{Q}_1[k],\dots,\mf{Q}_\mfs{N}[k])$ if the infimum is taken over the set of possible $\mf{Q}_i[k]$, which is compact by the boundedness property established in item 2) of Theorem \ref{th:main}. 

Henceforth, we take $\varepsilon<\min(\tilde{\varepsilon}, \overline{\varepsilon}(\tilde{\alpha}))$ and $K=\overline{K}(\tilde{\alpha})$. This allows $\|\mf{Q}_i[K]-\mf{Q}^*[K]\|< \tilde{\alpha}$, maintaining such property for all subsequent steps. The result follows by noting that $\tilde{\alpha}<\delta$ and that, if $\varepsilon=0$, the static case is recovered, optimal from Lemma \ref{lem:global:optim}.

\section{Application: distributed Kalman filter}\label{sec:application}

To further motivate the distributed discrete-time dynamic outer approximation of the intersection of ellipsoids, we exemplify how to exploit Algorithm~\ref{al:algorithm} to improve the mean square error performance of distributed Kalman filtering. To achieve this, we first need to ensure consistency, i.e., the fusion of predicted covariance matrices is such that the updated covariance matrices are still (tight) outer-approximations of the true covariance we would have if all the nodes tracked all the cross-correlations among nodes (see~\cite{Julier2017General} for further details). The next proposition proves consistency in the fusion of predicted covariances using Algorithm~\ref{al:algorithm}.
\begin{proposition}
\label{prop:fusion}
Let Assumptions~\ref{as:bounded}-\ref{as:unbounded} hold. Moreover, let 
$
\{\mf{x}[k]\}_{k\geq 0} 
$ be a Gaussian stochastic process and $\bar{\mf{x}}_i[k],$ $ i\in\mathcal{I}$ be Gaussian distributed correlated unbiased estimates for $\mf{x}[k]$ with covariance matrices given by $\mf{P}_i[k]$ and unknown correlations. Let $\mf{Q}^*[k]$ the solution to \eqref{eq:prob} with its corresponding weights $\{\lambda_i^*\}_{i\in\mathcal{I}}$ and
\begin{equation}
\label{eq:fusion}
\hat{\mf{x}}[k] = \mf{Q}^*[k]^{-1}\sum_{j\in\mathcal{I}}\lambda_j^*\mf{P}_j[k]^{-1}\bar{\mf{x}}_j[k].
\end{equation}
Then, the covariance matrix $\cov\{\hat{\mf{x}}[k]-\mf{x}[k]\}\preceq \mf{Q}^*[k]^{-1}$ always.
\end{proposition}
\begin{proof}
    Assumption \ref{as:unbounded} implies $\mf{Q}^*[k] = \sum_{j\in\mathcal{I}}\lambda_j^*\mf{P}_j[k]^{-1}$ due to $\mf{Q}^*[k]\in\mfs{rebdr}(\mathcal{C}^*[k])$. This means that \eqref{eq:fusion} is the standard covariance intersection fusion rule, whose consistency was proven in \cite{Niehsen2002}.
\end{proof}
{Proposition \ref{prop:fusion} exploits consistency to guarantee that the estimates at each node can be fused when cross-correlations are unknown, either for static of dynamic input matrices.} The previous result motivates estimating \eqref{eq:fusion} in a distributed way:
\begin{equation}
    \label{eq:fusion2}
    \mf{Q}_i[k]\hat{\mf{x}}_i[k] = \frac{1}{\mfs{N}}\sum_{j\in\mathcal{I}}\left(\frac{\lambda_\mf{P}^j[k]}{1-\sum_{l\in\mathcal{N}_j}\lambda_l^j[k]}\right)\mf{P}_j[k]^{-1}\bar{\mf{x}}_j[k]{,}
\end{equation}
where the right hand side of \eqref{eq:fusion2} can be computed using standard dynamic consensus tools with local inputs $\left(\frac{\lambda_\mf{P}^j[k]}{1-\sum_{l\in\mathcal{N}_j}\lambda_l^j[k]}\right)\mf{P}_j[k]^{-1}\bar{\mf{x}}_j[k]$.
\begin{proposition}
\label{prop:fusion2}
    Consider the same setting as in Proposition \ref{prop:fusion}. Moreover, let ${\Bar{\theta}} = 1$. Then, the update rule in eq.~\eqref{eq:fusion} leads to 
    $
    \lim_{k\to\infty}\|\hat{\mf{x}}_i[k] - \hat{\mf{x}}[k]\|= 0.
    $
\end{proposition}
\begin{proof}
    First, let 
$
\lambda^*_i = \frac{1}{\mfs{N}}\left(\frac{\lambda_\mf{P}^i}{1-\sum_{j\in\mathcal{N}_i}\lambda_j^i}\right),
$
where $\lambda^i_\mf{P}$ and $\lambda_j^i$ come from \eqref{eq:algo} in equilibrium at agent $i\in\mathcal{I}$. Then,
\begin{equation}
\label{eq:bound}
\mf{Q}[k]^*\preceq \sum_{j=1}^\mfs{N} \lambda_j^*\mf{P}_j[k]^{-1}.
\end{equation}
and $\sum_{j=1}^\mfs{N} \lambda_j^*\leq 1$ with $\lambda_j^*\geq 0$. Moreover, if Assumption \ref{as:unbounded} is satisfied, \eqref{eq:bound} is complied with equality.
Now, note that Lemma \ref{lem:global:optim} implies that in equilibrium all agents are in consensus at $\mf{Q}_i[k]=\mf{Q}^*$ with constant $\mf{Q}^*=\mf{Q}^*[K]$. Moreover, due to the constraints in \eqref{eq:algo}:
$
\mf{Q}_i[k] \kern -0.1cm = \mf{Q}^* \preceq   \lambda^i_\mf{P}\mf{P}_i[k]^{-1} + \sum_{j\in\mathcal{N}_i}\kern -0.1cm\lambda_j^i\mf{Q}^*$ implies $\mf{Q}^*\kern -0.1cm\left( - \sum_{j\in\mathcal{N}_i}\kern -0.1cm\lambda_j^i\kern -0.1cm\right)\kern -0.1cm\preceq \kern -0.1cm\lambda^i_\mf{P}\mf{P}_i[k]^{-1}.
$
Now, dividing by $1 - \sum_{j\in\mathcal{N}_i}\lambda_j^i$ and applying a summation over all agents in the previous relation:
$
\mfs{N}\mf{Q}^* \preceq \sum_{i=1}^\mfs{N}\left(\frac{\lambda_\mf{P}^i}{1 - \sum_{j\in\mathcal{N}_i}\lambda_j^i}\right)\mf{P}_i[k]^{-1}
$
which reduces to \eqref{eq:bound} by the definition of $\lambda_i^*$. Now, note that 
$
0\leq \lambda_\mf{P}^i\leq 1 - \sum_{j\in\mathcal{N}_i}\lambda_j^i
$
implying $\lambda_i^*\geq 0$ and as a result that $\lambda_i^* \leq 1/\mfs{N}$.
Hence, $\sum_{j=1}^\mfs{N}\lambda_j^*\leq 1$. In addition, if Assumption \ref{as:unbounded} is satisfied, the all previous inequalities are interchanged with equalities.

As a result, the proof of the proposition follows by noting that \eqref{eq:bound} is satisfied with equality as well as Theorem \ref{th:main} ensuring $\lim_{k\to\infty}\|\mf{Q}_i[k]-\mf{Q}^*[k]\|=0$.
\end{proof}

The results of Proposition~\ref{prop:fusion} and \ref{prop:fusion2} lead to a novel distributed Kalman filter with the following four steps: at instant $k$ and given $\hat{\mathbf{x}}_i[k-1], \hat{\mathbf{P}}_i[k-1]$, each node (i) predicts $\bar{\mathbf{x}}_i[k], \bar{\mathbf{P}}_i[k]$ using the known linear stochastic dynamics of the target system \mbox{$\mathbf{x}[k] = \mathbf{A}\mathbf{x}[k-1] + \mathbf{w}[k]$}, with $\mathbf{w}[k]$ a zero-mean Gaussian noise with covariance $\mathbf{W}$; (ii) exchanges of $\bar{\mathbf{x}}_i[k], \bar{\mathbf{P}}_i[k], \mathbf{Q}_i[k-1]$ with neighbors $i \in \mathcal{N}_i$; (iii) uses Algorithm~\ref{al:algorithm} to obtain $\mathbf{Q}_i[k]$, which is employed as the new covariance matrix to be updated; (iv) updates the predictions using \eqref{eq:fusion} and the Kalman filter update equations using the measurements $\mathbf{y}_i[k] = \mathbf{H}_i\mathbf{x}_i[k-1] + \mathbf{v}_i[k]$, with $\mathbf{v}_i[k]$ a zero-mean Gaussian noise with covariance $\mathbf{V}_i$, obtaining $\hat{\mathbf{x}}_i[k], \hat{\mathbf{P}}_i[k]$. We refer to \cite{Olfati2007DKF,Sebastian2021CDC} for the explicit expressions for the prediction and correction steps of distributed Kalman filtering. {It is interesting to remark that, even though Eq. \eqref{eq:fusion} is similar to a batch version of CI, it is not since it incorporates information from all nodes through $\mf{Q}^*[k], \lambda_j^*$, which are computed using Algorithm \ref{al:algorithm}. Hence, while existing methods are limited to local CI, our method can calculate a global CI once Algorithm \ref{al:algorithm} has converged.}

\section{Illustrative examples}\label{sec:examples}

\subsection{Constant inputs}\label{subsec:static_example}

First, we evaluate Algorithm~\ref{al:algorithm} in the static case, i.e., $\mathbf{P}_i[k] = \mathbf{P}_i$ $\forall i \in \mathcal{I}$. We generate a random connected graph $\mathcal{G}$ of $\mathsf{N}=6$ nodes, leading to the following edge set: \mbox{${\mathcal{F}}=\{(1,2), (1,3), (1,5), (2,4), (2,5), (2,6)\}$}. Input matrices are of dimension $n=2$, and they are initialized as $\mathbf{P}_i = \mathbf{L}_i^{\top}\mathbf{L}_i$, where the elements of each $\mathbf{L}_i$ are randomly generated using a uniform distribution between $-1$ and $1$. The resulting ellipses are characterized by the following matrices: 
$$
\begin{aligned}
    \mathbf{P}_1^{-1} \kern -0.13cm=\kern -0.13cm
    &
    \begin{pmatrix}
        4.6 & -3.8
        \\
        -3.8 & 4.2
    \end{pmatrix}
    &\kern -0.3cm
    \mathbf{P}_2^{-1} \kern -0.13cm=\kern -0.13cm 
    &
    \begin{pmatrix}
        1.5 & -0.2
        \\
        -0.2 & 2.0
    \end{pmatrix}
    &\kern -0.3cm
    \mathbf{P}_3^{-1} \kern -0.13cm=\kern -0.13cm 
    &
    \begin{pmatrix}
        9.5 & 0.4
        \\
        0.4 & 2.3
    \end{pmatrix}
    \\
    \mathbf{P}_4^{-1} \kern -0.13cm =\kern -0.13cm 
    &
    \begin{pmatrix}
        2.8 & -2.2
        \\
        -2.2 & 4.5
    \end{pmatrix}
    &\kern -0.3cm
    \mathbf{P}_5^{-1} \kern -0.13cm =\kern -0.13cm  
    &
    \begin{pmatrix}
        11.0 & 7.9
        \\
        7.9 & 6.7
    \end{pmatrix}
    &\kern -0.3cm
    \mathbf{P}_6^{-1} \kern -0.13cm =\kern -0.13cm
    &
    \begin{pmatrix}
        11.5 & -3.9
        \\
        -3.9 & 3.1
    \end{pmatrix}
\end{aligned}
$$
Since the input ellipsoids are constant, we set $\bar{\theta} = 1$. Also, let \mbox{$q^i_1[k] \leq q^i_2[k] \leq \hdots \leq q^i_n[k] $} the sorted eigenvalues of $\mathbf{Q}_i[k]$ and $q^*_1[k] \leq q^*_2[k] \leq \hdots \leq q^*_n[k] $ the sorted eigenvalues of $\mathbf{Q}^*[k]$.

\begin{figure}
    \centering
\includegraphics[width=0.9\columnwidth]{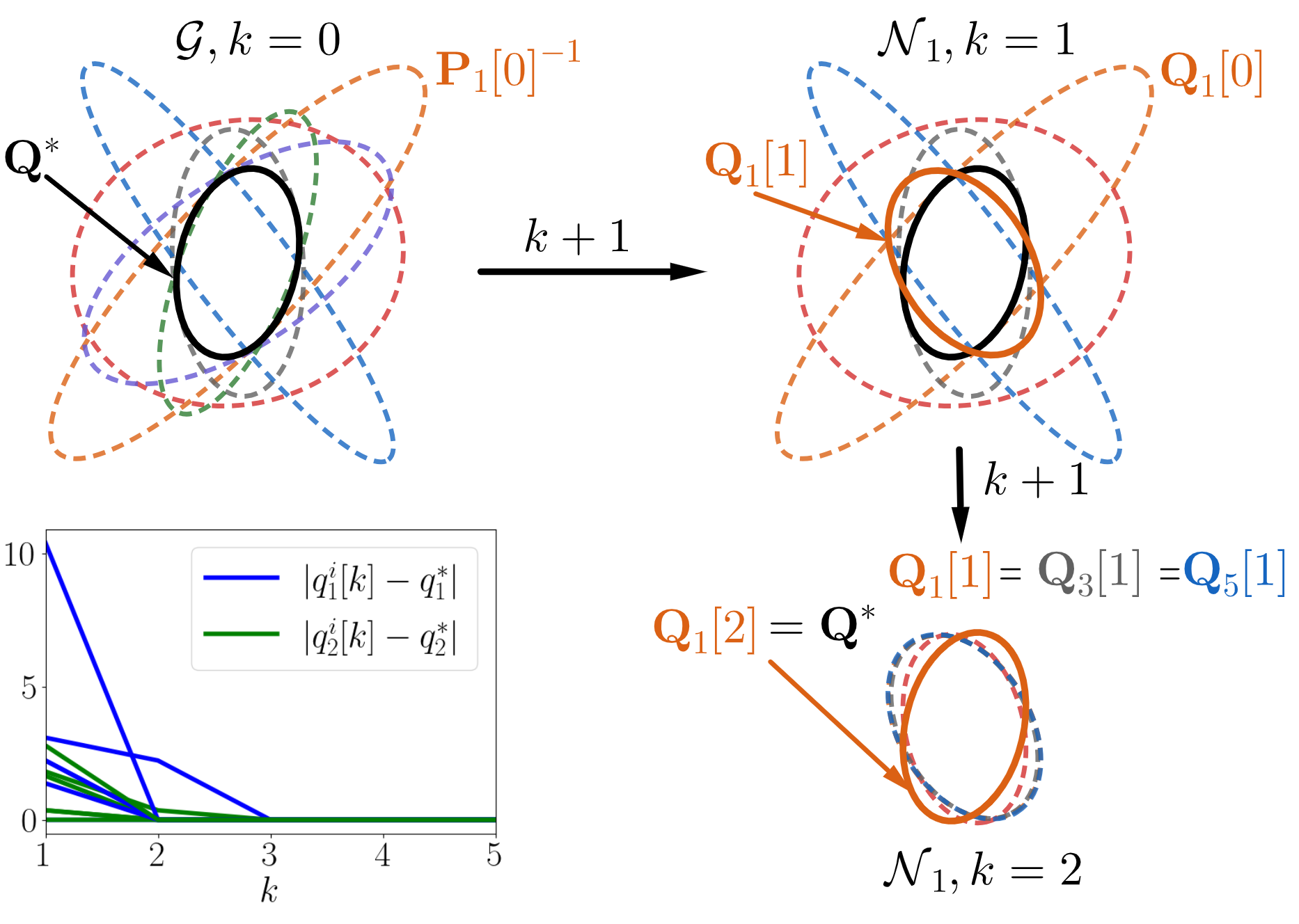}
    \caption{Illustrative example of Algorithm~\ref{al:algorithm} applied to a static problem. At each time step, the error between the eigenvalues of the global optimum matrix and the local estimates $|q^i_j[k] - q^*_j|$ decreases until the error becomes zero in finite time (in this case, $K=3$). The top left pannel shows the initial ellipsoids at each node and the optimal solution from the original centralized L\"owner-John method. The top right pannel shows, for node $1$, its estimate after using Algorithm \ref{al:algorithm} at $k=1$ and the ellipsoids exchanged with its neighbors. The bottom right pannel depicts the estimate at node $1$ and $k=2$ after using again Algorithm \ref{al:algorithm}, depicting that node $1$ already recovers the desired $\mathbf{Q}^*$. The bottom left pannel depicts the evolution of the error with time for all the nodes.}
    \label{fig:static_example}
\end{figure}

Fig.~\ref{fig:static_example} shows the evolution of the estimates at each node with time. To measure the difference between $\mathbf{Q}_i[k]$ and $\mathbf{Q}^*$ we compute the absolute error between their sorted eigenvalues as $|q^i_j[k] - q^*_j|$ $\forall j \in \{1, \hdots, n\}$ and $\forall i \in \mathcal{I}$. We can observe how the estimates converge in finite time to the global optimum of the centralized outer L\"owner-John method, with $K=3$. By focusing on node $i=1$, we can see how the estimates are refined at each time step, reaching the global optimum at $k=2$. The value of $K$ depends on the topology and how the input ellipsoids are across the network. However, for any case, the global minimum is perfectly attained in finite time.

\subsection{Dynamic case}\label{subsec:dynamic_example}

For the dynamic case where the input matrices evolve with time, we use the same graph from Subsection~\ref{subsec:static_example}. We use the same initial ellipsoids of Section~\ref{subsec:static_example} and apply a different oscillatory rotation to each of them at each time step. In particular, $\mathbf{P}_i[k] = (1 + \phi_i[k])\mathbf{R}_i[k]^{\top}\mathbf{P}_i[0]\mathbf{R}_i[k]$ where $\mathbf{R}_i[k]$ is a 2D rotation matrix with rotation angle $\psi_i[k] = A \sin(\omega_i k)$, and $\phi_i[k] = B \sin (\omega_i k)$. We choose $A = \frac{\pi}{25}$, $B = \frac{1}{200}$ and $\omega_i \in [1,2]$. It can be shown that $\bar{\theta}=\frac{1}{0.98}$ complies Assumption \ref{as:bounded}.  

\begin{figure}
    \centering
    \begin{tabular}{c}
    \includegraphics[width=1\columnwidth]{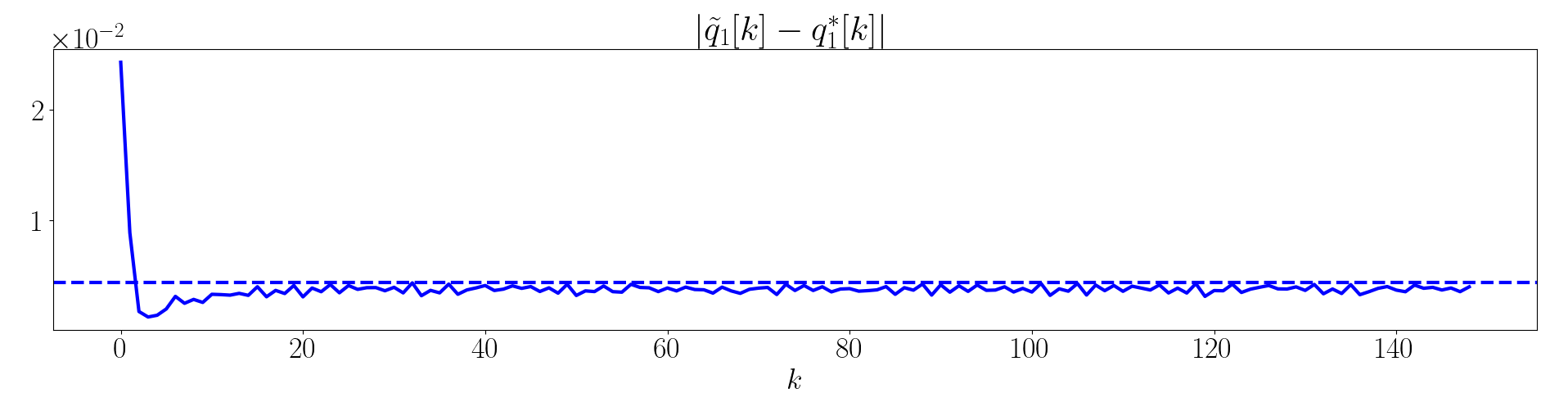}
         \\
    \includegraphics[width=1\columnwidth]{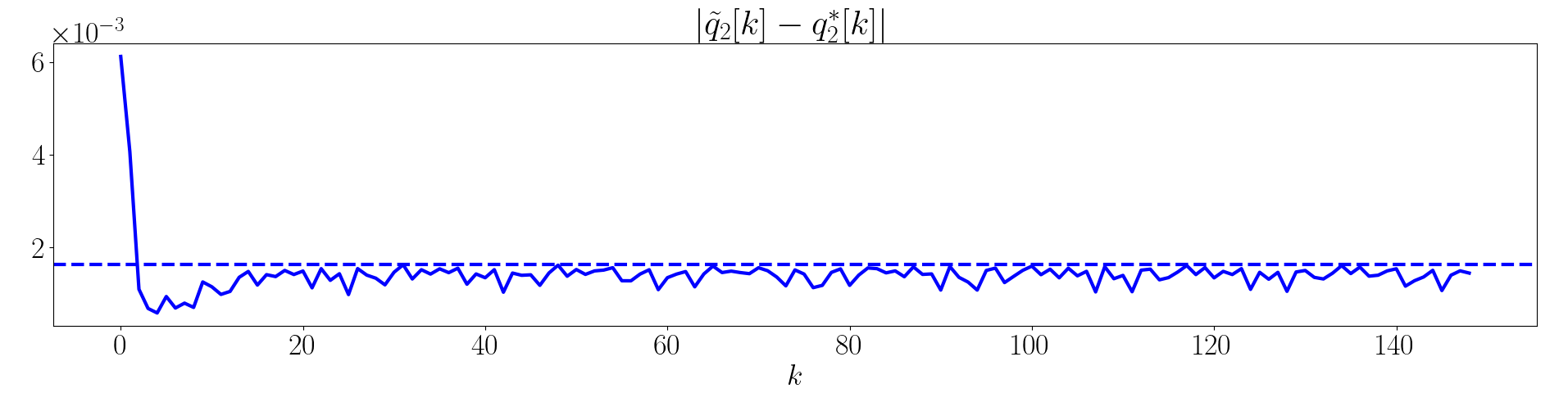}
    \end{tabular}
    \caption{Illustrative example of Algorithm~\ref{al:algorithm} applied to a dynamic problem. The plot depicts the time evolution of the error between the eigenvalues of the global optimum matrix and the average of the local estimates $|\tilde{q}_j[k] - q^*_j|$, for $j\in\{1,2\}$.}
    \label{fig:dynamic_example}
\end{figure}

Fig.~\ref{fig:dynamic_example} depicts the evolution of the error between the eigenvalues of the global optimum matrix and the average of the local estimates $|\tilde{q}_j[k] - q^*_j|$ with time, where $\tilde{q}_j[k] = \frac{1}{\mathsf{N}}\sum_{i=1}^{\mathsf{N}}q^i_j[k]$. It is observed the error trajectories converge in finite-time to some bounded region around the global optimum after a transient, which is maintained invariant for the rest of the experiment, consistent with the sub-optimality result in Theorem \ref{th:main}. % The size of the region is reduced for values of $\bar{\theta}$ closer to $1$. %The difference in convergence speed is observable from $k=0$ to $k=20$: for $\frac{1}{\bar{\theta}}=0.8$ the algorithm takes $4$ less time steps to converge compared to $\frac{1}{\bar{\theta}}=0.98$. The difference in convergence speed makes it more significant when the diameter of the graph increases or when the difference between input matrices at different instants is larger, as near $k=200$, where Algorithm~\ref{al:algorithm} with $\frac{1}{\bar{\theta}}=0.98$ does not have time to converge.

\subsection{Application example: distributed Kalman filtering}\label{subsec:dkf}

% By using Algorithm~\ref{al:algorithm} to track the best ellipsoidal approximation of the intersection of the updated covariance matrices, we are obtaining the best approximation of the region inside which the true fused covariance matrix under known correlation lies. Therefore, we expect improved mean square performance compared to other distributed Kalman filters.

We compare the our novel distributed Kalman filter developed in Section~\ref{sec:application} with the consensus distributed Kalman filter in \cite{Olfati2007DKF} (CDKF) and the equivalent centralized Kalman filter that we obtain by collecting all the measurements acquired by each node at a central server, fusing them and using a standard Kalman filter \cite{Kalman1960Kalman}.
The former serves as a comparison with an establish distributed Kalman filter whereas the latter serves as a baseline. We run $100$ simulations with randomly generated connected graphs of $\mathsf{N}=50$ nodes. The target system is the same in \cite{Sebastian2021CDC}, and the ellipsoids are of $n=4$. We replicate the experimental setting reported in \cite{Sebastian2021CDC}, choosing $\bar{\mathbf{x}}_i[0], \bar{\mathbf{P}}_i[0]$ randomly. We set $\mathbf{W} = 2\times10^{-8} \mathbf{I}$ and $\mathbf{V}_i = \mu_i \mathbf{I}$ with $\mu_i$ drawn from a uniform distribution between $0.03$ and $0.05$.

\begin{figure}
    \centering
    \includegraphics[width=1\columnwidth]{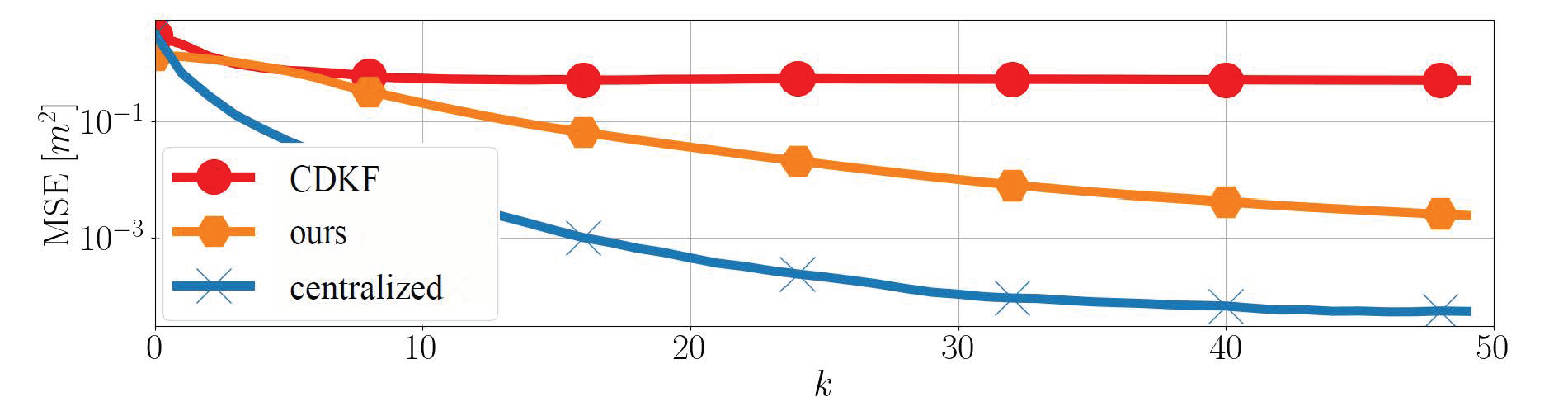}
    \caption{Evolution of the averaged mean square error with time for our novel distributed Kalman filter, the CDKF and the equivalent centralized Kalman filter. }
    \label{fig:dkf_result}
\end{figure}

Fig.~\ref{fig:dkf_result} reports the average of the mean square error across nodes and simulation runs for the two distributed Kalman filters:
$
\text{MSE}[k] = \frac{1}{100}\sum_{s=1}^{100}\frac{1}{\mathsf{N}}\sum_{i=1}^{\mathsf{N}}||\bar{\mathbf{x}}^s_i[k] - \mathbf{x}^s[k]||^2.
$ 
Our proposed algorithm surpasses CO-DKF by exploiting a better approximation of the true fused covariance matrix under known correlations: while our approach tracks the global minimum across all the nodes, CO-DKF only computes the outer approximation of the intersection of the neighboring ellipsoids. The difference between both approximations is greater for larger and sparser networks. 

\section{Conclusions}\label{sec:conclusions}

This work presented the first distributed and discrete-time algorithm to compute the tightest outer ellipsoid that approximates the intersection of a set of $\mathsf{N}$ ellipsoids distributed across a network. We reformulated the centralized outer L\"owner-John method to a local semi-definite program that exploits the neighboring information. In particular, by exchanging and scaling the neighboring estimates of the global optimum, the algorithm tracks the global minimum in finite time and perfect accuracy in static problems. In dynamic problems, the proposal tracks the global minima in finite time and with bounded accuracy. As a figure of merit, the proposed algorithm was applied in a distributed Kalman filtering task, demonstrating superior performance than the state-of-the-art by exchanging one additional matrix with neighbors. Nevertheless, the proposed algorithm can be used for distributed identification, estimation, and control problems.

%%%%%%%%%%%%%%           
% REFERENCES %
%%%%%%%%%%%%%%

\bibliographystyle{IEEEtran}
\bibliography{biblio}

\balance

\end{document}